\documentclass[a4paper, 12pt]{amsart}

\usepackage[a4paper,margin=25mm]{geometry}  
\usepackage{amsmath}
\usepackage{amssymb}
\usepackage{graphicx}
\usepackage{amsthm}
\usepackage{caption}
\usepackage{hyperref}
\usepackage[initials]{amsrefs}

\BibSpec{article}{%
  +{}{\PrintAuthors} {author}
  +{,}{ \textit} {title}
  +{,}{ } {journal}
  +{}{ \textbf} {volume}
  +{}{ \parenthesize} {date} 
  +{,}{ } {pages}
  +{,}{ } {note}
  +{.}{} {transition}
  +{}{ } {review}
}

\BibSpec{misc}{%
  +{}{\PrintAuthors} {author}
  +{,}{ \textit} {title}
  +{,}{ } {note}
  +{}{ \parenthesize} {date}
  +{.}{} {transition}
  +{}{ } {review}
}

\makeatletter
    
    \@addtoreset{equation}{section}
  \makeatother
\theoremstyle{plain}
\newtheorem{thm}{Theorem}[section]
\newtheorem*{thm*}{Theorem A}
\newtheorem*{thm**}{Theorem B}
\newtheorem{prop}[thm]{Proposition}

\newtheorem{lem}[thm]{Lemma}
\theoremstyle{definition}

\newtheorem{df}[thm]{Definition}
\newtheorem{ex}[thm]{Example}

\newcommand{\R}{\mathbb{R}}

\newcommand{\tr}{\mathop{\mathrm{tr}}\nolimits}

\newcommand{\rad}{\mathop{\mathrm{rad}}\nolimits}
\newcommand{\codim}{\mathop{\mathrm{codim}}\nolimits}

\newcommand{\Ker}{\mathop{\mathrm{Ker}}\nolimits}
\newcommand{\Hom}{\mathop{\mathrm{Hom}}\nolimits}

\newcommand{\ad}{\mathop{\mathrm{ad}}\nolimits}

\newcommand{\Der}{\mathop{\mathrm{Der}}\nolimits}

\newcommand{\dsum}{\displaystyle\sum}

\newcommand{\ddelta}{\mathop{\scriptstyle{\Delta}}\nolimits}

\def\vector#1{\mbox{\boldmath $#1$}} 

\renewcommand{\tilde}{\widetilde}

\begin{document}
\title{Non-degenerate bilinear forms and left-symmetric structures}
\author{Naoki Kato}

\maketitle

\begin{abstract}
Chu proved that a symplectic structure $\omega$ on an even-dimensional Lie algebra induces a left-symmetric structure which is defined by $\omega(x\ddelta_{\omega} y,z)=-\omega(y,[x,z])$.
El Bourkadi and Mansouri proved that a cosymplectic structure on an odd-dimensional Lie algebra $\mathfrak{g}$ induces the left-symmetric structure, which is defined by using a linear isomorphism from $\mathfrak{g}$ to $\mathfrak{g}^*$ associated with the cosymplectic structure.
In this paper, for a non-degenerate bilinear form $\phi$ on a Lie algebra $\mathfrak{g}$, we give a necessary and sufficient condition for the product $\ddelta_{\phi}$ on $\mathfrak{g}$ defined by $\phi(x\ddelta_{\phi} y,z)=-\phi(y,[x,z])$ to be a left-symmetric structure.
We also prove that the left-symmetric structure $\ddelta_{\phi}$ is complete if and only if the Lie algebra $\mathfrak{g}$ is unimodular.
Moreover, we formulate the notion of double extension for non-degenerate bilinear forms and prove that a certain class of non-degenerate bilinear forms is obtained by double extension.
\end{abstract}

\section{Introduction}
For a connected Lie group $G$, a left-invariant flat and torsion-free affine connection on $G$ is called a left-invariant affine structure on $G$.
In the case where $G$ is simply connected, it is known that the Lie group $G$ admits a left-invariant affine structure if and only if its Lie algebra $\mathfrak{g}$ admits a left-symmetric structure.

A non-degenerate closed $2$-form on a Lie algebra $\mathfrak{g}$ is called a symplectic structure on $\mathfrak{g}$ and a Lie algebra equipped with a symplectic structure is called a symplectic Lie algebra.
It is known that a unimodular Lie algebra admitting a symplectic structure is solvable (\cite{Chu74}).
A classification of low-dimensional Lie algebras which admit symplectic structures has been given in \cite{Ova06} and \cite{KhaGozMed04}.

Chu \cite{Chu74} proved that a symplectic structure $\omega$ on a Lie algebra $\mathfrak{g}$ induces a left-symmetric structure on $\mathfrak{g}$.
Chu also proved that semisimple Lie algebras do not admit left-symmetric structures.
It follows that semisimple Lie algebras do not admit symplectic structures.
Some properties of the induced left-symmetric structures are given in \cite{DarMed96} and \cite{MedRev91}.

For a cosymplectic structure $(\alpha, \omega)$ on $\mathfrak{g}$, El Bourkadi and Mansouri \cite{BouMan24Lef}  proved that a product $\ddelta_{(\alpha,\omega)}$ on $\mathfrak{g}$ defined by
$$
\Phi_{(\alpha,\omega)}(x\ddelta_{(\alpha,\omega)} y)(z)=-\Phi_{(\alpha,\omega)}(y)([x,z])
$$
is a left-symmetric structure on $\mathfrak{g}$, where $\Phi_{(\alpha,\omega)}\colon  \mathfrak{g} \to \mathfrak{g}^*$ is the associated linear isomorphism defined by $\Phi_{(\alpha,\omega)}(x)= i_{x}\omega+\alpha(x)\alpha$.
They also gave a classification of cosymplectic structures on $3$- and $5$-dimensional Lie algebras.
For $3$-dimensional Lie algebras, they gave explicit descriptions of left-symmetric structures induced by cosymplectic structures (Corollary 4.3 in \cite{BouMan24Lef}).

In this paper, we consider a bilinear form $\phi\colon\mathfrak{g}\times\mathfrak{g}\to{\R}$ on a Lie algebra $\mathfrak{g}$ which satisfies the condition 
$$\phi_A([x,y],z)+\phi_A([y,z],x)+\phi_A([z,x],y)=\phi_S(x,[y,z]),
$$
where $\phi_S$ and $\phi_A$ are the symmetric part and the skew-symmetric part of $\phi$, respectively.
We call a non-degenerate bilinear form satisfying this condition an NDB structure and a Lie algebra equipped with an NDB structure an NDB Lie algebra.
For an NDB structure $\phi$ on $\mathfrak{g}$, we can define a product $\ddelta_\phi$ by $\phi(x\ddelta_{\phi} y,z)=-\phi(y,[x,z])$.
This definition is analogous to the definitions of the products in the symplectic and cosymplectic cases.
We prove that the product $\ddelta_{\phi}$ is a left-symmetric structure on $\mathfrak{g}$ (Theorem \ref{thm1}).
We also prove that the left-symmetric structure $\ddelta_{\phi}$ induced by an NDB structure $\phi$ is complete if and only if $\mathfrak{g}$ is unimodular, which is a generalization of Theorem 1.2 in \cite{MedRev91} for left-symmetric structures induced by symplectic structures.

The class of left-symmetric structures induced by NDB structures contains the class of left-symmetric structures induced by symplectic structures and cosymplectic structures (see Subsection \ref{subsection_cosymplectic}).
We show that this class properly contains the class of left-symmetric structures induced by symplectic structures (see Example \ref{ex1} and Subsection \ref{ex2}).

A Lie algebra equipped with an $\ad$-invariant non-degenerate symmetric form is called a quadratic Lie algebra.
Medina and Revoy \cite{MedRev85} proved that every $n$-dimensional non-abelian indecomposable solvable quadratic Lie algebra is obtained from an $(n-2)$-dimensional quadratic Lie algebra by an extension.
This method of constructing $n$-dimensional quadratic Lie algebras from $(n-2)$-dimensional quadratic Lie algebras is called quadratic double extension (see \cite{MedRev85}).
Bordemann gave a generalization of this method in \cite{Bor97}, which is called $T^*$-extension.

The notion of double extension for symplectic Lie algebras was introduced in \cite{MedRev91} and generalized in \cite{DarMed96}.
Medina and Revoy \cite{MedRev91} proved that every nilpotent symplectic Lie algebra is obtained from the zero Lie algebra by successive double extensions. 
El Bourkadi and Mansouri introduced the notion of double extension for Lie algebras equipped with cosymplectic structures and gave a condition under which a cosymplectic structure on a Lie algebra is obtained by double extension.

In Section \ref{section4}, we give several constructions of NDB structures using extension procedures.
In Subsection \ref{subsection_extension}, we prove that a certain extension of an NDB Lie algebra by a symplectic Lie algebra admits an NDB structure (see Proposition \ref{prop2}).
Conversely, we prove that an NDB structure satisfying a certain condition is obtained by this extension (Theorem \ref{thm3}).

In Subsection \ref{subsection_double_extension}, we formulate the notion of double extension for NDB Lie algebras and prove that every $n$-dimensional NDB Lie algebra which has a certain kind of central element is obtained from an $(n-2)$-dimensional NDB Lie algebra (see Definition \ref{def_double_extension} and Theorem \ref{thm_double_extension}).

\section{Preliminaries}
\subsection{Left-symmetric structures}
 Let $\mathfrak{g}$ be a finite-dimensional  Lie algebra over ${\R}$.
\begin{df}
A bilinear product $\ddelta \colon  \mathfrak{g}\times \mathfrak{g}\to\mathfrak{g}$ is said to be a left-symmetric structure on $\mathfrak{g}$ if $\ddelta$ satisfies the following equations:
\begin{align}
x\ddelta(y\ddelta z)-(x \ddelta y)\ddelta z&= y\ddelta (x\ddelta z)-(y\ddelta x)\ddelta z \quad \text{and} \label{df1-1}\\
[x,y]&=x\ddelta y-y\ddelta x\label{df1-2}
\end{align}
for any $x,y,z\in \mathfrak{g}$.

Two left-symmetric structures $\ddelta$ on $\mathfrak{g}$ and $\ddelta'$ on  $\mathfrak{g}'$ are said to be isomorphic if there exists an isomorphism of Lie algebras $\phi\colon\mathfrak{g}\to\mathfrak{g}'$ such that $\phi(x\ddelta y)=\phi(x)\ddelta' \phi(y)$ for any $x,y\in\mathfrak{g}$.
 
\end{df}

We note that equation \eqref{df1-1} is equivalent to the equation $(x,y,z)=(y,x,z)$, where $(x,y,z)$ is the associator of $\ddelta$ defined by $(x,y,z)=x\ddelta(y\ddelta z)-(x\ddelta y)\ddelta z$.

\begin{df}
A left-symmetric structure $\ddelta $ on $\mathfrak{g}$ is said to be complete if $r(x)\in\mathfrak{gl(g)}$ is nilpotent for any $x\in\mathfrak{g}$, where $r\colon\mathfrak{g}\to\mathfrak{gl(g)}$ is the right multiplication map of $\ddelta$ defined by $r(x)(y)=y\ddelta x$.
\end{df}

Left-symmetric structures on $\mathfrak{g}$ correspond to flat and torsion-free affine connections on the simply connected Lie group $G$ with Lie algebra $\mathfrak{g}$.
The completeness of $\ddelta$ corresponds to the completeness of the corresponding connection (see \cite{Hel79} and \cite{Kim86}).
 
\subsection{Symplectic structures and cosymplectic structures}

A non-degenerate $2$-form $\omega$ satisfying
\begin{equation}\label{def_symplectic}
\omega([x,y],z)+\omega([y,z], x)+\omega([z,x],y)=0
\end{equation}
is called a symplectic structure (or symplectic form) on $\mathfrak{g}$.
The equation \eqref{def_symplectic} means that $\omega$ is a closed $2$-form, that is, a $2$-cocycle in the Chevalley--Eilenberg complex with trivial coefficients.
Chu \cite{Chu74} proved that a product $\ddelta_{\omega} \colon \mathfrak{g}\times \mathfrak{g}\to \mathfrak{g}$ defined by $\omega(x\ddelta_{\omega} y,z)=-\omega(y,[x,z])$ is a left-symmetric structure on $\mathfrak{g}$.
The left-symmetric structure $\ddelta_{\omega}$ is complete if and only if $\mathfrak{g}$ is unimodular (Theorem 1.2 in \cite{MedRev91}).

For a symplectic structure $\omega$, we define the associated linear isomorphism $\Phi_{\omega}\colon \mathfrak{g}\to \mathfrak{g}^*$ by $\Phi_{\omega}(x)=\iota_x\omega$, where $\iota_x$ is the interior product by $x$.
Then the left-symmetric structure $\ddelta_\omega$ can be written as $x\ddelta_{\omega}  y=\Phi_{\omega}^{-1} (\ad^*(x)\Phi_\omega(y))$, where $\ad^*\colon\mathfrak{g}\to\mathfrak{gl}(g)$ is the coadjoint representation of $\mathfrak{g}$.

Let $\ell\colon\mathfrak{g}\to \mathfrak{gl}(\mathfrak{g})$ be the left multiplication map of $\ddelta_{\omega}$, which is defined by $\ell(x)(y)=x\ddelta_{\omega} y$.
Then $\ell$ can be written as 
$$\ell(x)=\Phi_{\omega}^{-1}\circ \ad^*(x)\circ \Phi_{\omega}.
$$

In analogy with the symplectic case, El Bourkadi and Mansouri \cite{BouMan24Lef} proved that cosymplectic structures define left-symmetric structures.

\begin{df}
Let $\alpha$ and $\omega$ be a $1$-form and a $2$-form on a $(2n+1)$-dimensional Lie algebra $\mathfrak{g}$, respectively.
A pair $(\alpha, \omega)$ is said to be a cosymplectic structure on $\mathfrak{g}$ if $\alpha$ and $\omega$ are closed forms and satisfy $\alpha\wedge \omega^n\not= 0$.
\end{df}

Define a linear map $\Phi_{(\alpha,\omega)}\colon \mathfrak{g}\to\mathfrak{g}^*$ by
$$
\Phi_{(\alpha,\omega)}(x) =\alpha(x)\alpha +\iota_x \omega
$$
and a product $\ddelta_{(\alpha,\omega)}$ on $\mathfrak{g}$ by
$$\Phi_{(\alpha,\omega)}(x\ddelta_{(\alpha,\omega)}y)(z)=-\Phi_{(\alpha,\omega)}(y)([x,z]).
$$
Then El Bourkadi and Mansouri proved that this product is a left-symmetric structure (see Theorem 2.4 in \cite{BouMan24Lef}).
The left multiplication map of this left-symmetric structure can also be written as 
$$\ell(x)=\Phi_{(\alpha,\omega)}^{-1}\circ \ad^*(x)\circ \Phi_{(\alpha,\omega)}.
$$

\section{Non-degenerate bilinear forms and left-symmetric structures}\label{section_non-dege}

\subsection{Left-symmetric structures induced by non-degenerate bilinear forms}

For a finite-dimensional vector space $V$ over ${\R}$, we naturally identify the space of bilinear forms $\Hom(V\otimes V, {\R})$ with $V^*\otimes V^*$.
The space $V^*\otimes V^*$ is decomposed as
$$
V^*\otimes V^*=S^2 V^*\oplus \bigwedge\nolimits^2 V^*,
$$
where $S^2V^*$ is the space of symmetric bilinear forms and $\bigwedge^2 V^*$ is the space of skew-symmetric bilinear forms on $V$.
For a bilinear form $\phi\in \Hom(V\otimes V, {\R})$, we denote by $\phi_S$ and $\phi_A$ the symmetric part and the skew-symmetric part, respectively.

\begin{df}
A bilinear form $\phi\in\Hom(V\otimes V, {\R})$ is said to be non-degenerate if the map 
$$
\phi^{\flat}_L\colon V\to V^*, \quad x\mapsto \phi(x,\, \cdot\, )
$$
is an isomorphism of vector spaces.
\end{df}
 
We remark that $\phi^{\flat}_L$ is an isomorphism if and only if the map
$$
\phi^{\flat}_R\colon V\to V^*, \quad x\mapsto \phi(\, \cdot \, , x)
$$
is an isomorphism.

Let $\mathfrak{g}$ be a Lie algebra and $\phi\in\Hom(\mathfrak{g}\otimes \mathfrak{g},{\R})$ a non-degenerate bilinear form on $\mathfrak{g}$.
We define a product $\ddelta_{\phi}$ on $\mathfrak{g}$ by 
\begin{equation}
\phi(x\ddelta_{\phi} y, z)=-\phi(y,[x,z]).  \label{df3}
\end{equation}
The left multiplication map associated with $\ddelta_{\phi}$ can be written as $\ell(x)=(\phi_L^{\flat})^{-1}\circ \ad^*(x)\circ \phi_L^{\flat}$ for each $x\in\mathfrak{g}$.

\begin{thm}\label{thm1}
The product $\ddelta_{\phi}$ is a left-symmetric structure on $\mathfrak{g}$ if and only if the symmetric part $\phi_S$ and the skew-symmetric part $\phi_A$ of $\phi$ satisfy
\begin{equation}
\phi_A([x,y],z)+\phi_A([y,z],x)+\phi_A([z,x],y)=\phi_S(x,[y,z]) \label{eq2}
\end{equation}
for any $x,y,z\in\mathfrak{g}$.
\end{thm}

We remark that the equation \eqref{eq2} implies the $\ad$-invariance of $\phi_S$:
\begin{equation}
\phi_S([x,y],z)+\phi_S(y,[x,z])=0. \label{eq1}
\end{equation}
Indeed, 
\begin{align*}
&\phi_S([x,y],z)+\phi_S(y,[x,z])\\
 &= \phi_A([z,x],y)+\phi_A([x,y],z)+\phi_A([y,z],x)+\phi_A([y,x],z)+\phi_A([x,z],y)+\phi_A([z,y],x)\\
&=0.
\end{align*}

For an $\ad$-invariant symmetric form $B\in S^2\mathfrak{g}^*$, we can define $J_B\in \bigwedge^3\mathfrak{g}^*$ by $J_B(x,y,z)=B(x,[y,z])$.
The map $J\colon B\mapsto J_B$ is the dual of the Koszul map and $J_B$ is a $3$-cocycle (see \cite{Cor16}).
The equation \eqref{eq2} means that the $3$-cocycle $J_{\phi_S}$ satisfies $J_{\phi_S}=-d\phi_A$.
In particular, $J_{\phi_S}$ is a $3$-coboundary.

\begin{lem}\label{lem1}
The product $\ddelta_{\phi}$ is a left-symmetric structure on $\mathfrak{g}$ if and only if $\phi$ satisfies
\begin{equation}
\phi([x,y],z)+\phi(y,[x,z])=\phi(x,[y,z]). \label{eq3}
\end{equation}
\end{lem}
\begin{proof}
For $x,y,z\in \mathfrak{g}$, we have
$$
\phi(x\ddelta_{\phi} y-y\ddelta_{\phi} x,z)=-\phi(y,[x,z])+\phi(x,[y,z]).
$$
Hence the equation \eqref{df1-2} implies the equation \eqref{eq3}.

Conversely, since $\phi$ is non-degenerate, the equation \eqref{eq3} implies the equation \eqref{df1-2}.
 
Suppose that the equation \eqref{eq3} holds.
Then, for $x,y,z,w\in\mathfrak{g}$, we have 
\begin{align*}
&\phi((x,y,z)-(y,x,z),w)\\
&=\phi(x\ddelta_{\phi}(y\ddelta_{\phi} z),w)-\phi((x\ddelta_{\phi} y)\ddelta_{\phi} z,w)-\phi(y\ddelta_{\phi}(x\ddelta_{\phi} z),w)+\phi((y\ddelta_{\phi} x)\ddelta_{\phi} z,w)\\
&=-\phi(y\ddelta_{\phi} z,[x,w])+\phi(z, [x\ddelta_{\phi} y, w])+\phi(x\ddelta_{\phi} z,[y,w]) -\phi(z,[y\ddelta_{\phi} x,w])\\
&=\phi(z,[y,[x,w]])-\phi(z,[x,[y,w]])+\phi(z,[x\ddelta_{\phi} y-y\ddelta_{\phi} x,w])\\
&= \phi(z,[y,[x,w]])+\phi(z,[x,[w,y]])+\phi(z,[w,[y,x]])\\
&=0.
\end{align*}
Therefore, $\ddelta_{\phi}$ satisfies the equation \eqref{df1-1}.

\end{proof}

\begin{proof}[{\bf Proof of Theorem \ref{thm1}}]

By Lemma \ref{lem1}, it suffices to prove that $\phi$ satisfies the equation \eqref{eq3} if and only if $\phi_S$ and $\phi_A$ satisfy the equation \eqref{eq2}.

If $\phi_S$ and $\phi_A$ satisfy the equation \eqref{eq2}, then $\phi$ satisfies
\begin{align*}
\phi([x,y],z)+\phi(y,[x,z])
&=\phi_A([x,y],z)+\phi_A(y,[x,z])+\phi_S([x,y],z)+\phi_S(y,[x,z])\\
&=-\phi_A([y,z],x)+\phi_S(x,[y,z])+\phi_S([x,y],z)+\phi_S(y,[x,z])\\
&=\phi_A(x,[y,z])+\phi_S(x,[y,z])\\
&=\phi(x,[y,z])
\end{align*}
since $\phi_S$ is $\ad$-invariant.

Conversely, suppose that $\phi$ satisfies the equation \eqref{eq3}.
By equations
\begin{align*}
\phi(x,[y,z])&=\phi([x,y],z)+\phi(y,[x,z]) \quad \text{ and}\\
\phi(z,[x,y])&=\phi([z,x],y)+\phi(x,[z,y]),
\end{align*}
we obtain 
\begin{align}
&\phi_S(x,[y,z])+ \phi_A(x,[y,z])\label{eq4} \\
&=\phi_S([x,y],z)+\phi_S(y,[x,z])+\phi_A([x,y],z)+\phi_A(y,[x,z])\quad  \text{ and} \notag\\
& \phi_S(z,[x,y])+\phi_A(z,[x,y]) \label{eq5}\\
&=\phi_S([z,x],y)+\phi_S(x,[z,y])+\phi_A([z,x],y)+\phi_A(x,[z,y]). \notag
\end{align}

Since $\phi_S$ is symmetric and $\phi_A$ is skew-symmetric, the equation \eqref{eq5} is equivalent to the equation
$$
\phi_S([x,y],z)-\phi_A([x,y],z)=-\phi_S(y,[x,z])-\phi_S(x,[y,z])+\phi_A(y,[x,z])-\phi_A(x,[y,z]).
$$
By comparing this equation with the equation \eqref{eq4}, we have
\begin{equation}
\phi_S(y,[x,z])+\phi_S([x,y],z)=0. \label{eq6}
\end{equation}
By substituting this equation  into the equation \eqref{eq4}, we obtain the equation \eqref{eq2}. 
\end{proof}

\begin{df}
We call a non-degenerate bilinear form $\phi$ on a Lie algebra $\mathfrak{g}$ which satisfies the equation \eqref{eq2} an NDB structure on $\mathfrak{g}$ and call a Lie algebra equipped with an NDB structure an NDB Lie algebra.
\end{df}

If an NDB structure $\phi$ satisfies $\phi_A=0$, then the equation \eqref{eq2} implies that $\mathfrak{g}$ is abelian.
If  $\phi$ satisfies $\phi_S=0$, then $\phi_A$ is a non-degenerate closed $2$-form, that is, $\phi_A$ is a symplectic structure.
Hence the notion of NDB structures is a generalization of that of symplectic structures.
In Subsection \ref{subsection_cosymplectic}, we describe the relationship between NDB structures and cosymplectic structures.

We give examples of NDB structures and the left-symmetric structures induced by them. 
\begin{ex}\label{ex1}
Let $\mathfrak{g}=\mathfrak{aff}({\R})$ be the Lie algebra of the group of affine transformations of ${\R}$, which is the $2$-dimensional non-abelian Lie algebra.
We consider $\mathfrak{g}$ as a subalgebra of $\mathfrak{gl}({\R}^2)$ and put $t=\begin{pmatrix}1&0\\0&0\end{pmatrix}$ and $x=\begin{pmatrix}0&1\\0&0\end{pmatrix}$.
Then $\mathfrak{g}={\R}\langle t, x\rangle$ and $[t,x]=x$.

Let $\{t^*,x^*\}$ be the dual basis of $\{t,x\}$.
Then $\phi_S=t^*\otimes t^*$ is an $\ad$-invariant symmetric form, which is the Killing form of $\mathfrak{g}$.
Since $[\mathfrak{g},\mathfrak{g}]={\R}\langle x\rangle$, $\phi_S$ satisfies $\phi_S(\mathfrak{g}, [\mathfrak{g}, \mathfrak{g}])=0$.
Hence the equation \eqref{eq2} means that $\phi_A$ is a closed form.

Put $\phi_A=t^*\wedge x^*=t^*\otimes x^*-x^*\otimes t^*$.
Then $\phi=t^*\otimes t^*+(t^*\otimes x^*-x^*\otimes t^*)$ is non-degenerate and satisfies the equation \eqref{eq2}.

The linear isomorphism $\phi_L^{\flat}\colon\mathfrak{g}\to\mathfrak{g}^*$ associated with $\phi$ is given by $\phi_L^{\flat}(t)=t^*+x^*$ and $\phi_L^{\flat}(x)=-t^*$.
The coadjoint representation $\ad^*\colon\mathfrak{g}\to\mathfrak{gl}(\mathfrak{g}^*)$ is represented by 
$$
\ad^*(t)=\begin{pmatrix} 0&0\\0&-1\end{pmatrix} \quad \text{and} \quad \ad^*(x)=\begin{pmatrix} 0&1\\0&0\end{pmatrix}
$$
with respect to the basis $\{t^*,x^*\}$.
Then the left multiplication map of the left-symmetric structure $\ddelta_{\phi}$ is given by
\begin{align*}
\ell^{\phi}(t)&=(\phi_L^{\flat})^{-1}\circ \ad^*(t)\circ \phi_L^{\flat}=\begin{pmatrix}-1&0\\-1&0\end{pmatrix} \quad \text{and}\\
\ell^{\phi}(x)&=(\phi_L^{\flat})^{-1}\circ \ad^*(x)\circ \phi_L^{\flat}=\begin{pmatrix} 0&0\\-1&0\end{pmatrix}.
\end{align*}	

We note that $\phi_A$ is a symplectic structure on $\mathfrak{aff}({\R})$ and the left multiplication map $\ell^{\phi_A}$ of the left-symmetric structure $\ddelta_{\phi_A}$ induced by $\phi_A$ is given by
$$
\ell^{\phi_A}(t)=\begin{pmatrix} -1&0\\0&0\end{pmatrix} \quad \text{and} \quad \ell^{\phi_A}(x)=\begin{pmatrix} 0&0\\-1&0\end{pmatrix}.
$$
The left-symmetric structure $\ddelta_{\phi}$ is isomorphic to $\mathcal{F}(-1)$ and the left-symmetric structure $\ddelta_{\phi_A}$ coincides with the structure $\mathcal{A}_2$ in the list in \cite{MedSalGir16}.
In particular, $\ddelta_{\phi}$ is not isomorphic to $\ddelta_{\phi_A}$ (see Section 3 of \cite{MedSalGir16}).
\end{ex}

\subsection{An example of NDB structures on $T^*\mathfrak{h}_3$}\label{ex2}

Let $\mathfrak{h}_3$ be the $3$-dimensional Heisenberg Lie algebra and let $\mathfrak{g}=T^*\mathfrak{h}_3$ be the cotangent Lie algebra of $\mathfrak{h}_3$.
Let $\{t,e_1,e_2\}$ be a basis of $\mathfrak{h}_3$ with $[t,e_2]=e_1$ and $\{z,e^1,e^2\}$ the dual basis of $\mathfrak{h}_3^*$.
Put $f_1=e^2, f_2=e^1$.
Then $\mathcal{B}=\{t,z,e_1,f_1,e_2,f_2\}$ is a basis of $T^*\mathfrak{h}_3$.
The Lie bracket of $T^*\mathfrak{h}_3$ is given by
$$
[t,e_2]=e_1,\quad 
[t,f_2]=-f_1, \quad \text{and}\quad 
[e_2,f_2]=z.
$$

For a Lie algebra $\mathfrak{g}$, the cotangent Lie algebra $T^*\mathfrak{g}$ admits an $\ad$-invariant non-degenerate symmetric form 
$$\phi_S(x_1+f_1, x_2+f_2)=f_1(x_2)+f_2(x_1),
$$
where $x_i\in\mathfrak{g}$ and $f_i\in\mathfrak{g}^*$ for $i=1,2$ (see Example 1 in \cite{Ova 16}).

For $\mathfrak{g}=\mathfrak{h}_3$, the $\ad$-invariant non-degenerate symmetric form $\phi_S$ on $T^*\mathfrak{h}_3$ can be written as
$$
\phi_S=(z^*\otimes t^*+t^*\otimes z^*)+(e_1^*\otimes f_2^*+f_2^*\otimes e_1^*)+(e_2^*\otimes f_1^*+f_1^*\otimes e_2^*),
$$
where $\{t^*,z^*,e_1^*,f_1^*,e_2^*,f_2^*\}$ is the dual basis  of $\mathcal{B}$.

By a direct computation, we have 
\begin{align*}
&J_{\phi_S}=t^*\wedge e_2^*\wedge f_2^* \quad \text{and}\\
&d(\alpha t^*\wedge z^*+\beta e_1^*\wedge f_2^*+(-1-\alpha+\beta)e_2^*\wedge f_1^* )=-t^*\wedge e_2^*\wedge f_2^*,
 \end{align*}
 where $\alpha,\beta\in{\R}$.
Hence 
$$
\phi^{(\alpha,\beta)}=\phi_S+\alpha t^*\wedge z^*+\beta e_1^*\wedge f_2^*+(-1-\alpha+\beta)e_2^*\wedge f_1^*
$$
satisfies the equation \eqref{eq2}.
The bilinear form $\phi^{(\alpha,\beta)}$ is non-degenerate if and only if $\alpha$ and $\beta$ satisfy $\alpha\not= \pm 1, \beta\not= \pm 1$, and $\alpha-\beta\not= 0, -2$.

Suppose that $\alpha, \beta$ satisfy $\alpha\not= \pm 1, \beta\not= \pm 1$, and $\alpha-\beta\not= 0, -2$.
The left-symmetric structure $\ddelta_{\phi^{(\alpha,\beta)}}$ on $T^*\mathfrak{h}_3$ induced by the NDB structure $\phi^{(\alpha, \beta)}$ is given by the following nonzero products:
\begin{alignat*}{2}
t \ddelta_{\phi^{(\alpha,\beta)}} e_2&=\dfrac{-\alpha+\beta}{1+\beta}e_1.  &\qquad 
t \ddelta_{\phi^{(\alpha,\beta)}} f_2&=\dfrac{-1+\beta}{2+\alpha-\beta} f_1. \\
e_2 \ddelta_{\phi^{(\alpha,\beta)}} t&=\dfrac{-1-\alpha}{1+\beta} e_1.   &\qquad
e_2 \ddelta_{\phi^{(\alpha,\beta)}} f_2&=\dfrac{1-\beta}{1-\alpha} z.\\
f_2 \ddelta_{\phi^{(\alpha,\beta)}} t&=\dfrac{1+\alpha}{2+\alpha-\beta} f_1. & \qquad 
f_2 \ddelta_{\phi^{(\alpha,\beta)}} e_2 &=\dfrac{\alpha-\beta}{1-\alpha} z.
\end{alignat*}

We put $\overline{\alpha}_i=\dfrac{-\alpha_i+\beta_i}{1+\beta_i}$ and $\overline{\beta}_i=\dfrac{-1+\beta_i}{2+\alpha_i-\beta_i}$ for $i=1,2$.
We remark that, for $n=2$, Theorem 3.2 in \cite{Kat25-2} remains valid without conditions (3.3) and (3.4) in \cite{Kat25-2}.
Suppose that $\alpha_i$ and $\beta_i$ satisfy $-2\alpha_i+\beta_i\not= 1$ and $\alpha_i+\beta_i\not= 0$.
Then $\overline{\alpha}_i$ and $\overline{\beta}_i$ satisfy the remaining hypotheses of Theorem 3.2 in \cite{Kat25-2}.
Hence, by Theorem 3.2 in \cite{Kat25-2}, $\ddelta_{\phi^{(\alpha_1,\beta_1)}} $ is isomorphic to $\ddelta_{\phi^{(\alpha_2, \beta_2)}}$  if and only if $(\overline{\alpha}_1,\overline{\beta}_1)$ coincides with either $(\overline{\alpha}_2, \overline{\beta}_2)$ or $(-\overline{\beta}_2, -\overline{\alpha}_2)$.
This condition is equivalent to the condition that $(\alpha_1,\beta_1)$ coincides with either $(\alpha_2,\beta_2)$ or  $(\alpha_2, \alpha_2-\beta_2+1)$.

Therefore, the family of left-symmetric structures $\ddelta_{\phi^{(\alpha,\beta)}}$ obtained from NDB structures $\phi^{(\alpha,\beta)}$ contains infinitely many pairwise non-isomorphic left-symmetric structures on $T^*\mathfrak{h}_3$.

We prove that the family of left-symmetric structures $\ddelta_{\phi^{(\alpha,\beta)}}$ contains infinitely many left-symmetric structures which are not isomorphic to any left-symmetric structure induced by a symplectic structure on $T^*\mathfrak{h}_3$.
The Lie algebra $T^*\mathfrak{h}_3$ admits the following symplectic structures (see No.\ 18 in the list in \cite{KhaGozMed04}):
\begin{align*}
\omega_1(\lambda)&=t^*\wedge z^*-\lambda e_2^*\wedge f_1^* -(\lambda -1)e_1^*\wedge f_2^* \quad (\lambda\not= 0,1).\\
\omega_2(\lambda)&=\lambda t^* \wedge z^* -t^*\wedge f_1^* + 2\lambda e_1^*\wedge f_2^* +e_2^*\wedge z^*+ \lambda e_2^*\wedge f_1^*  \quad (\lambda \not= 0).\\
\omega_3&=-t^*\wedge z^*-2e_1^*\wedge f_2^*-e_2^*\wedge f_1^*+f_1^*\wedge f_2^*.
\end{align*}
Any symplectic structure on $T^*\mathfrak{h}_3$ is isomorphic to one of the above symplectic structures (Theorem 5 in \cite{KhaGozMed04}).

The  matrix representations of the associated linear isomorphisms $\Phi_{\omega_1(\lambda)}, \Phi_{\omega_2(\lambda)}$, and $\Phi_{\omega_3}$ of $\omega_1(\lambda), \omega_2(\lambda)$, and $\omega_3$ with respect to the basis $\mathcal{B}$ and its dual basis are given by the following:
\begin{alignat*}{2}
\Phi_{\omega_1(\lambda)}&=\begin{pmatrix}0&-1&0&0&0&0\\1&0&0&0&0&0\\0&0&0&0&0&\lambda -1 \\0&0&0&0&-\lambda &0\\0&0&0&\lambda &0&0\\0&0&-(\lambda -1)&0&0&0\end{pmatrix}.
&\quad 
\Phi_{\omega_2(\lambda)}&=\begin{pmatrix}0&-\lambda&0&1&0&0\\ \lambda &0&0&0&1&0\\0&0&0&0&0&-2\lambda \\ -1&0&0&0&\lambda &0\\0&-1&0&-\lambda &0&0\\0&0&2\lambda &0&0&0\end{pmatrix}.\\
\Phi_{\omega_3}&=\begin{pmatrix}0&1&0&0&0&0\\-1&0&0&0&0&0\\ 0&0&0&0&0&2\\0&0&0&0&-1&-1\\ 0&0&0&1&0&0\\0&0&-2&1&0&0\end{pmatrix}.
\end{alignat*}
Then the left-symmetric structures $\ddelta_{\omega_1(\lambda)}, \ddelta_{\omega_2(\lambda)}$, and $\ddelta_{\omega_3}$ are given by the following nonzero products:
\begin{alignat*}{2}
 t \ddelta_{\omega_1(\lambda)} e_2&=\dfrac{\lambda}{\lambda-1}e_1.  &\qquad 
t \ddelta_{\omega_1(\lambda)} f_2&=-\dfrac{\lambda -1}{\lambda} f_1. \\ 
e_2 \ddelta_{\omega_1(\lambda)} t&=\dfrac{1}{\lambda-1} e_1.   &\qquad
e_2 \ddelta_{\omega_1(\lambda)} f_2&=-(\lambda -1) z.\\
f_2 \ddelta_{\omega_1(\lambda)} t&= \dfrac{1}{\lambda}f_1. & \qquad 
f_2 \ddelta_{\omega_1(\lambda)} e_2 &=-\lambda z. \\[1em]
t \ddelta_{\omega_2(\lambda)} t&= -\dfrac{1}{2\lambda } e_1 &\qquad
t \ddelta_{\omega_2(\lambda)} e_2&=\dfrac{1}{2}e_1.  \\ 
t \ddelta_{\omega_2(\lambda)} f_2&=-\dfrac{2\lambda^2}{\lambda^2+1} f_1-\dfrac{2\lambda}{\lambda^2+1} z. &\qquad
e_2 \ddelta_{\omega_2(\lambda)} t&=-\dfrac{1}{2} e_1.   \\
e_2 \ddelta_{\omega_2(\lambda)} e_2&= -\dfrac{1}{2\lambda} e_1 &\qquad
e_2 \ddelta_{\omega_2(\lambda)} f_2&=-\dfrac{2\lambda}{\lambda^2+1} f_1+\dfrac{2\lambda^2}{\lambda^2+1} z.\\
f_2 \ddelta_{\omega_2(\lambda)} t&= \dfrac{1-\lambda^2}{\lambda^2+1}f_1 -\dfrac{2\lambda}{\lambda^2+1} z.  &\qquad 
f_2 \ddelta_{\omega_2(\lambda)} e_2 &=-\dfrac{2\lambda}{\lambda^2+1}f_1+\dfrac{\lambda^2-1}{\lambda^2+1} z.\\[1em]
t \ddelta_{\omega_3} e_2&=\dfrac{1}{2}e_1.  &\qquad 
t \ddelta_{\omega_3} f_2&=-\dfrac{1}{2} e_1-2f_1. \\
e_2 \ddelta_{\omega_3} t&=-\dfrac{1}{2} e_1.   &\qquad
e_2 \ddelta_{\omega_3}f_2&= 2 z.\\
f_2 \ddelta_{\omega_3} t&= -\dfrac{1}{2}e_1 -f_1. & \qquad 
f_2 \ddelta_{\omega_3} e_2 &=z.\\
f_2 \ddelta_{\omega_3} f_2 &=z. 
\end{alignat*}

If $\lambda\not= -1,2$, then $\alpha_\lambda=\dfrac{\lambda}{\lambda -1}$  and $\beta_\lambda=-\dfrac{\lambda-1}{\lambda} $ satisfy the hypotheses of Theorem 3.2 in \cite{Kat25-2}.
For any $\lambda \not= 0,1$, we have 
$$
\left(\dfrac{-\alpha+\beta}{1+\beta}, \dfrac{-1+\beta}{2+\alpha-\beta}\right) \not= \left(\dfrac{\lambda}{\lambda -1}, -\dfrac{\lambda-1}{\lambda} \right), \left(\dfrac{\lambda-1}{\lambda} , -\dfrac{\lambda}{\lambda -1} \right).
$$
Hence, by Theorem 3.2 in \cite{Kat25-2}, $\ddelta_{\phi^{(\alpha,\beta)}}$ is not isomorphic to $\ddelta_{\omega_1(\lambda)}$ if $-2\alpha+\beta\not= 1, \alpha+\beta\not= 0,$ and  $\lambda\not= -1,2$.

\begin{prop}
The left-symmetric structure $\ddelta_{\phi^{(\alpha,\beta)}}$ is isomorphic to neither $\ddelta_{\omega_2(\lambda)}$ nor $\ddelta_{\omega_3}$.
\end{prop}

\begin{proof}
Let $F\colon T^*\mathfrak{h}_3\to T^*\mathfrak{h}_3$ be an automorphism of the Lie algebra $T^*\mathfrak{h}_3$.
Put $w_1=t, w_2=z, w_3=e_1,w_4=f_1,w_5=e_2,$ and $w_6=f_2$ and let $(a_{ij})$ be the matrix representation of $F$ with respect to the basis $\{w_1, \ldots , w_6\}$.
Since $F$ preserves the center of the Lie algebra $T^*\mathfrak{h}_3$, we have $a_{ij}=0$ for $i=1,5,6$ and $j=2,3,4$.

First, we suppose that  $F$ satisfies $F(x\ddelta_{\phi^{(\alpha,\beta)}} y)=F(x)\ddelta_{\omega_2(\lambda)} F(y)$ for any $x,y\in T^*\mathfrak{h}_3$.
Since $w_j\ddelta_{\phi^{(\alpha,\beta)}} w_j=0$, we have
$$
F(w_j)\ddelta_{\omega_2(\lambda)} F(w_j)=F(w_j\ddelta_{\phi^{(\alpha,\beta)}} w_j)=0.
$$
The $e_1$-component of $F(w_j)\ddelta_{\omega_2(\lambda)} F(w_j)$ is $-\dfrac{1}{2\lambda}(a_{1j}^2+a_{5j}^2)$.
Hence $a_{1j}= a_{5j}=0$ for $j=1,\ldots ,6$.
This contradicts the invertibility of $(a_{ij})$.

Next, we suppose that $F(x\ddelta_{\phi^{(\alpha,\beta)}} y)=F(x)\ddelta_{\omega_3} F(y)$.
Since $w_j\ddelta_{\phi^{(\alpha,\beta)}} w_j=0$, we have
\begin{align*}
0
&=F(w_j\ddelta_{\phi^{(\alpha,\beta)}} w_j)\\
&=F(w_j)\ddelta_{\omega_3} F(w_j)\\
&=-a_{1j}a_{6j} e_1 -3a_{1j}a_{6j} f_1 +(3a_{5j}a_{6j}+a_{6j}^2)z.
\end{align*}
Hence equations
\begin{equation}
a_{1j}a_{6j}=0 \quad \text{and} \quad  a_{6j}(3a_{5j}+a_{6j})=0 \label{eqex1}
\end{equation}
hold for any $j=1,\ldots ,6$.

By the definition of $\ddelta_{\omega_3}$ and the Lie bracket of $T^*\mathfrak{h}_3$, we obtain the following equations:
\begin{align}
F(t)\ddelta_{\omega_3} F(e_2)
&=\dfrac{1}{2}\ (a_{11}a_{55}- a_{11}a_{65}-a_{51}a_{15}-a_{61}a_{15})e_1 \label{eqex2}\\ 
&\quad +(-2a_{11}a_{65}-a_{61}a_{15})f_1 +(2a_{51}a_{65}+a_{61}a_{55}+a_{61}a_{65}) z.\notag \\
F(t)\ddelta_{\omega_3}F(f_2)
&=\dfrac{1}{2}(a_{11}a_{56}-a_{11}a_{66}-a_{51}a_{16}-a_{61}a_{16})e_1\label{eqex3}\\
&\quad +(-2a_{11}a_{66}-a_{61}a_{16})f_1+(2a_{51}a_{66}+a_{61}a_{56}+a_{61}a_{66})z. \notag \\
[F(t),F(e_2)]
&=(a_{11}a_{55}-a_{51}a_{15})e_1-(a_{11}a_{65}-a_{61}a_{15})f_1+(a_{51}a_{65}-a_{61}a_{55})z. \label{eqex4}\\
[F(t), F(f_2)]
&=(a_{11}a_{56}-a_{51}a_{16})e_1-(a_{11}a_{66}-a_{61}a_{16})f_1 +(a_{51}a_{66}-a_{61}a_{56})z. \label{eqex5}
\end{align}

Since the matrix $(a_{ij})$ is invertible, equations \eqref{eqex1} imply that one of the following three cases holds:
\begin{itemize}
\item[(a)] $a_{65}=a_{66}=a_{11}=0$ and $a_{61}=-3a_{51}$.
\item[(b)] $a_{61}=a_{66}=a_{15}=0$ and $a_{65}=-3a_{55}$.
\item[(c)] $a_{61}=a_{65}=a_{16}=0$ and $a_{66}=-3a_{56}$.
\end{itemize}
In case (a), equations \eqref{eqex2}, \eqref{eqex3}, \eqref{eqex4}, and \eqref{eqex5}  imply
\begin{align*}
[F(t), F(e_2)]&= -F(t)\ddelta_{\omega_3} F(e_2) \quad \text{and} \\
[F(t),F(f_2)]&=-F(t)\ddelta_{\omega_3} F(f_2).
\end{align*}
Since $[t,e_2]=e_1, [t,f_2]=-f_1, t\ddelta_{\phi^{(\alpha,\beta)}} e_2=\overline{\alpha} e_1,$ and $t\ddelta_{\phi^{(\alpha,\beta)}} f_2=\overline{\beta} f_1$, we have 
\begin{align*}
[F(t), F(e_2)]&= F([t,e_2])=F(e_1), \\
F(t)\ddelta_{\omega_3}F(e_2)&=F(t\ddelta_{\phi^{(\alpha,\beta)}} e_2)=\overline{\alpha}F(e_1),\\
[F(t), F(f_2)]&= F([t,f_2])=-F(f_1), \quad \text{and}\\
F(t)\ddelta_{\omega_3}F(f_2)&=F(t\ddelta_{\phi^{(\alpha,\beta)}} f_2)=\overline{\beta}F(f_1).
\end{align*}
Hence $\overline{\alpha}$ and $\overline{\beta}$ satisfy $(\overline{\alpha}, \overline{\beta})=(-1,1)$.
This is a contradiction since there do not exist $\alpha$ and $\beta$ satisfying $\overline{\alpha}=-1$ and $\overline{\beta}=1$. 

In cases (b) and (c), a similar argument yields $(\overline{\alpha}, \overline{\beta})=\left(2,-\dfrac{1}{2}\right)$ and $(\overline{\alpha}, \overline{\beta})=\left(\dfrac{1}{2}, -2\right)$,  respectively.
In each case, we obtain a contradiction.
\end{proof}

\subsection{Completeness of left-symmetric structures}
For a left-symmetric structure $\ddelta_{\omega}$ induced by a symplectic structure $\omega$, Medina and Revoy \cite{MedRev91} proved that $\ddelta_{\omega}$ is complete if and only if $\mathfrak{g}$ is unimodular.
For a left-symmetric structure induced by an NDB structure, a similar result holds.

\begin{thm}\label{complete}
Let $\phi$ be an NDB structure on $\mathfrak{g}$ and $\ddelta_{\phi}$ the left-symmetric structure defined by \eqref{df3}.
Then $\ddelta_{\phi}$ is complete if and only if $\mathfrak{g}$ is unimodular.

\end{thm}
\begin{proof}

It is known that a left-symmetric structure is complete if and only if $\tr(r(x))=0$ for every $x\in\mathfrak{g}$ (Theorem 3.31 in \cite{Bur06}, Proposition 1 in \cite{Gic05}).
Hence, for a left-symmetric structure induced by an NDB structure, it suffices to prove that $\tr(r(x))=0$ if and only if $\tr(\ad(x))=0$ for each $x\in\mathfrak{g}$.

We use the following elementary lemma from linear algebra.

\begin{lem}\label{lem2}
Suppose that $f, g\in\mathfrak{gl(g)}$ satisfy $\phi(f(x),y)=\phi(x,g(y))$ for any $x,y\in\mathfrak{g}$.
Then $\tr(f)=\tr(g)$.
\end{lem}
\begin{proof}
Let $\{e_1,\ldots, e_n\}$ be a basis of $\mathfrak{g}$ and let $A=(a_{ij})$ and $B=(b_{ij})$ be the matrix representations of $f$ and $g$, respectively.
Put $\phi_{ij}=\phi(e_i,e_j)$ and $C=(\phi_{ij})$.
Since $\phi$ is non-degenerate, $C$ is invertible.
Then we have
\begin{align*}
\phi(f(e_i), e_j)&=\dsum_{k=1}^n a_{ki} \phi_{kj} \quad \text{and}\\
\phi(e_i,g(e_j))&=\dsum_{k=1}^n b_{kj}\phi_{ik}
\end{align*}
for any $i,j=1,\ldots, n$.
Hence we obtain ${}^tAC=CB$.
Since $C$ is invertible, $\tr(A)=\tr(B)$.
\end{proof}

By the definition of $\ddelta_{\phi}$, we have $\phi(\ell(x)(y),z)=-\phi(y,\ad(x)(z))$.
Hence, by Lemma \ref{lem2},  $\tr(\ell(x))=-\tr(\ad(x))$ holds for each $x\in\mathfrak{g}$.
Since $\ad(x)=\ell(x)-r(x)$, we obtain $\tr(r(x))=2\tr(\ell(x))=-2\tr(\ad(x))$.
Therefore, $\tr(r(x))=0$ if and only if $\tr(\ad(x))=0$.

\end{proof}

\subsection{NDB structures and cosymplectic structures}\label{subsection_cosymplectic}

Let $(\alpha,\omega)$ be a cosymplectic structure on a $(2n+1)$-dimensional Lie algebra $\mathfrak{g}$.
Then  $\phi(x,y)=\alpha(x)\alpha(y)+\omega(x,y)$ is a non-degenerate bilinear form with $\phi_S=\alpha\otimes \alpha$ and $\phi_A=\omega$.
Since $\alpha$ and $\omega$ are closed forms, $\phi_S$ and $\phi_A$ satisfy the equation \eqref{eq2}.

Conversely, we show that an NDB structure $\phi=\phi_S+\phi_A$ on a Lie algebra $\mathfrak{g}$ induces a symplectic structure or cosymplectic structure on $\mathfrak{g}$ under certain conditions.
Consider the radical 
$$
\rad(\phi_S)=\{x\in\mathfrak{g}\mid \phi_S(x,\, \cdot \, )=0\}=\Ker((\phi_S)^{\flat}_L)
$$
of $\phi_S$, which is an ideal of $\mathfrak{g}$.
Then $\phi_S$ induces a non-degenerate $\ad$-invariant symmetric form $\phi'_S$ on the quotient Lie algebra $\mathfrak{g}'=\mathfrak{g}/\rad(\phi_S)$ defined by $\phi'_S(p(x), p(y))=\phi_S(x,y)$, where $p\colon\mathfrak{g}\to\mathfrak{g}'$ is the projection.

\begin{prop}\label{prop1}
Suppose that $\codim \rad(\phi_S)=1$.
\begin{itemize}
\item[(i)] If $\dim\Ker(\phi_A|_{\rad(\phi_S)})=0$, then there exists $d\in\mathfrak{g}$ such that $((\phi_S)^{\flat}_L(d), \phi_A)$ is a cosymplectic structure on $\mathfrak{g}$.
\item[(ii)] If $\dim \Ker(\phi_A|_{\rad(\phi_S)})=1$, then $\phi_A$ is a symplectic structure on $\mathfrak{g}$.
\end{itemize}
\end{prop}

\begin{proof}
By the assumption, there exists $d\in\mathfrak{g}$ such that $\mathfrak{g}=\rad(\phi_S)\oplus {\R}\langle d \rangle$ as a vector space.
Since $\rad(\phi_S)$ is an ideal of $\mathfrak{g}$, we obtain $[\mathfrak{g}, \mathfrak{g}]\subset\rad(\phi_S)$ and hence $(\phi_S)^{\flat}_L(d)$ is a closed form.
Then $\phi_A$ is a closed form by the equation \eqref{eq2}.

Suppose that $\dim\Ker(\phi_A|_{\rad(\phi_S)})=0$.
Then $\phi_A$ is non-degenerate on $\rad(\phi_S)$.
Hence $\phi_A|_{\rad(\phi_S)}$ is a symplectic form on $\rad(\phi_S)$.
Since $d\not\in\rad(\phi_S)$ and $\phi_S(\rad(\phi_S), d)=0$, we have $\phi_S(d,d)\not=0$.
Then $(\phi_S)^{\flat}_L(d)$ and $\phi_A$ satisfy $(\phi_S)^{\flat}_L(d)\wedge \phi_A^{m}\not=0$, where $2m$ is the dimension of $\rad(\phi_S)$.
Therefore $((\phi_S)^{\flat}_L(d), \phi_A)$ is a cosymplectic structure on $\mathfrak{g}$.

Suppose that $\dim\Ker(\phi_A|_{\rad(\phi_S)})=1$.
Take a nonzero element $e\in\Ker(\phi_A|_{\rad(\phi_S)})$ and a complementary subspace $V$ of ${\R}\langle e\rangle$ in $\rad(\phi_S)$.
Then $\rad(\phi_S)={\R}\langle e\rangle \oplus V$ and $\mathfrak{g}={\R}\langle e\rangle \oplus V \oplus {\R}\langle d \rangle$ as vector spaces.

Take a basis $\{e_1,\ldots ,e_k\}$ of $V$ and put $e_0=e, e_{k+1}=d$.
Then the matrix representation $(\phi(e_i,e_j))$ of $\phi$ with respect to the basis $\{e_0,\ldots ,e_{k+1}\}$ is of the form
$$
\begin{pmatrix}
0&{}^t \vector{0} & \phi_A(e,d)\\
\vector{0} & (\phi_A(e_i,e_j)) & (\phi_A(e_i,d))\\
\phi_A(d,e) & (\phi_A(d,e_j)) & \phi_S(d,d)
\end{pmatrix}.
$$
Since $\phi$ is non-degenerate, $\phi_A(e,d)\not= 0$ and $\phi_A|_V$ is non-degenerate.
Therefore $\phi_A$ is non-degenerate on $\mathfrak{g}$, that is, $\phi_A$ is a symplectic structure on $\mathfrak{g}$. 
\end{proof}

\section{Construction of NDB structures}\label{section4}
We have seen that symplectic and cosymplectic structures induce NDB structures.
In this section, we give several constructions of NDB structures by using extension procedures.

\subsection{Extension of NDB Lie algebras by symplectic Lie algebras}\label{subsection_extension}
Let $\omega$ be a symplectic structure on a $2r$-dimensional Lie algebra $\mathfrak{r}$ and let
$$
\Der(\mathfrak{r}, \omega)=\{A\in\Der(\mathfrak{r})\mid \omega(A(x), y)+\omega(x,A(y))=0\}
$$
be the set of symplectic derivations of $(\mathfrak{r}, \omega)$.
Let $\mathfrak{h}$ be a Lie algebra with an NDB structure $\phi'=\phi'_S+\phi'_A$, $D\colon\mathfrak{h}\to \Der(\mathfrak{r},\omega)$ a homomorphism, and $\mathfrak{g}=\mathfrak{h}\ltimes_D \mathfrak{r}$ the semidirect sum of $\mathfrak{h}$ and $\mathfrak{r}$ by $D$.
The Lie bracket on $\mathfrak{g}$ is given by
$$
[x+a, y+b]=[x,y]'+[a,b]_{\mathfrak{r}}+D(x)b-D(y)a
$$
for $x,y\in\mathfrak{h}, a,b\in\mathfrak{r}$, where $[\, \cdot \, , \, \cdot \, ]'$ and $[\, \cdot \, , \, \cdot \, ]_{\mathfrak{r}}$ are the Lie brackets on $\mathfrak{h}$ and $\mathfrak{r}$, respectively.  

We define a bilinear form $\phi$ on $\mathfrak{g}$ by 
\begin{equation}
\phi(x+a,y+b)=\phi'(x,y)+\omega(a,b). \label{eq4-1}
\end{equation}

\begin{prop}\label{prop2}
The bilinear form $\phi$ is an NDB structure on $\mathfrak{g}$.
\end{prop}

\begin{proof}
Since $\phi'$ and $\omega$ are non-degenerate on $\mathfrak{h}$ and $\mathfrak{r}$, respectively, $\phi$ is non-degenerate.

For $\xi_1=x+a, \xi_2=y+b, \xi_3=z+c\in\mathfrak{g}$, we have 
\begin{align*}
\phi_S(\xi_1,[\xi_2,\xi_3])
&=\phi_S(x+a, [y,z]'+[b,c]_{\mathfrak{r}}+D(y)c-D(z)b)\\
&=\phi'_S(x,[y,z]').
\end{align*}
Since  $\phi_A(\mathfrak{h},\mathfrak{r})=0$ and $d\omega=0$, we obtain
\begin{align*}
&\phi_A([\xi_1,\xi_2],\xi_3)+\phi_A([\xi_2,\xi_3],\xi_1)+\phi_A([\xi_3,\xi_1],\xi_2)\\
&=\phi'_A([x,y]',z)+\phi'_A([y,z]',x)+\phi'_A([z,x]',y)\\ &\quad +\omega([a,b]_{\mathfrak{r}}+D(x)b-D(y)a,c)+\omega([b,c]_{\mathfrak{r}}+D(y)c-D(z)b, a) \\ &\quad +\omega([c,a]_{\mathfrak{r}}+D(z)a-D(x)c,b)\\
&=\phi'_A([x,y]',z)+\phi'_A([y,z]',x)+\phi'_A([z,x]',y)\\
&=\phi'_S(x,[y,z]').
\end{align*}
Therefore the equation \eqref{eq2} holds.
\end{proof}

Conversely, an NDB structure satisfying certain conditions can be obtained from this construction.

\begin{thm}\label{thm3}
Let $\phi=\phi_S+\phi_A$ be an NDB structure on a Lie algebra $\mathfrak{g}$.
Suppose that there exists a subalgebra $\mathfrak{h}$ of  $\mathfrak{g}$ such that $\mathfrak{h}$ and $\rad(\phi_S)$ are orthogonal with respect to $\phi_A$ and $\mathfrak{g}=\mathfrak{h}\oplus\rad(\phi_S)$ as a vector space.

Then the following holds:
\begin{itemize}
\item[(i)] $\phi|_{\mathfrak{h}}=\phi_A|_{\mathfrak{h}}+\phi_S|_{\mathfrak{h}}$ is an NDB structure on $\mathfrak{h}$.
\item[(ii)] $\phi_A|_{\rad(\phi_S)}$ is a symplectic structure on $\rad(\phi_S)$.
\item[(iii)] There exists a homomorphism $D\colon\mathfrak{h}\to\Der(\rad(\phi_S),\phi_A|_{\rad(\phi_S)})$ such that $\mathfrak{g}$ is isomorphic to $\mathfrak{h}\ltimes_D\rad(\phi_S)$.
\item[(iv)] Under the identification of $\mathfrak{g}$ with $\mathfrak{h}\ltimes_D \rad(\phi_S)$,  $\phi$ coincides with the NDB structure defined by \eqref{eq4-1}.
\end{itemize}
\end{thm}

\begin{proof}
We write $\phi_A$ and $\phi_S$ as 
$$
\phi_A=\begin{pmatrix}\phi_A^{1,1} &\phi_A^{1,2}\\ \phi_A^{2,1} &\phi_A^{2,2}\end{pmatrix}
\quad \text{and} \quad 
\phi_S=\begin{pmatrix}\phi_S^{1,1} &\phi_S^{1,2}\\ \phi_S^{2,1} &\phi_S^{2,2}\end{pmatrix}
$$
with respect to the decomposition $\mathfrak{g}=\mathfrak{h}\oplus \rad(\phi_S)$.
By the definition of $\rad(\phi_S)$, $\phi_S^{1,2}=0, \phi_S^{2,1}=0$, and $\phi_S^{2,2}=0$.
Since $\mathfrak{h}$ and $\rad(\phi_S)$ are $\phi_A$-orthogonal, $\phi_A^{2,1}=0$ and $\phi_A^{1,2}=0$.
Hence we have
\begin{equation}
\phi=\phi_S+\phi_A=\begin{pmatrix}\phi_S^{1,1}+\phi_A^{1,1} & 0 \\ 0 & \phi_A^{2,2}\end{pmatrix}. \label{4-0}
\end{equation}

For any $\xi_1=x+a, \xi_2=y+b,\xi_3=z+c\in\mathfrak{g}=\mathfrak{h}\oplus\rad(\phi_S)$, since $\rad(\phi_S)$ is an ideal of $\mathfrak{g}$, we have
\begin{align*}
&\phi_S(\xi_1,[\xi_2,\xi_3])=\phi_S^{1,1}(x,[y,z]) \quad \text{and}\\
&\phi_A([\xi_1,\xi_2],\xi_3)+\phi_A([\xi_2,\xi_3],\xi_1)+\phi_A([\xi_3,\xi_1],\xi_2)\\
&=\phi_A^{1,1}([x,y],z)+\phi_A^{1,1}([y,z],x)+\phi_A^{1,1}([z,x],y)\\
&\quad +\phi_A^{2,2}([x,b]+[a,y]+[a,b],c)+\phi_A^{2,2}([y,c]+[b,z]+[b,c],a)\\ 
&\quad +\phi_A^{2,2}([z,a]+[c,x]+[c,a],b).
\end{align*}
By the equation \eqref{eq2}, these equations yield the following equations:
\begin{align}
&\phi_S^{1,1}(x,[y,z])=\phi_A^{1,1}([x,y],z)+\phi_A^{1,1}([y,z],x)+\phi_A^{1,1}([z,x],y). \label{4-1}\\
&d\phi_A^{2,2}=0. \label{4-2} \\
&\phi_A^{2,2}([x,b],c)+\phi_A^{2,2}([c,x],b)=0. \label{4-3}
\end{align}

Since $\phi$ is non-degenerate, $\phi_S^{1,1}+\phi_A^{1,1}$ and $\phi_A^{2,2}$ are non-degenerate.
Hence, by \eqref{4-1} and \eqref{4-2}, $\phi|_{\mathfrak{h}}$ is an NDB structure on $\mathfrak{h}$ and $\phi_A^{2,2}$ is a symplectic structure on $\rad(\phi_S)$.

The adjoint representation induces a homomorphism $D\colon \mathfrak{h}\to \Der(\rad(\phi_S))$ and $\mathfrak{g}$ is isomorphic to $\mathfrak{h}\ltimes_D \rad(\phi_S)$.
By \eqref{4-3}, $D(x)$ satisfies $\phi_A^{2,2}(D(x)b, c)+\phi_A^{2,2}(b,D(x)c)=0$.
Hence $D(x)\in \Der(\rad(\phi_S), \phi_A^{2,2}|_{\rad(\phi_S)})$.

(iv) follows from \eqref{4-0}.
\end{proof}

Since $\tr(D(x))=0$ for any $x\in\mathfrak{h}$, the left-symmetric structure associated with $\phi=\phi'+\omega$ is complete if and only if both $\mathfrak{h}$ and $\rad(\phi_S)$ are unimodular  by Theorem \ref{complete}.

\subsection{Double extension for  NDB structures}\label{subsection_double_extension}

For quadratic Lie algebras, the notion of double extension was introduced in \cite{MedRev85}.
For symplectic Lie algebras, the notion of double extension was introduced in \cite{MedRev91} and generalized in \cite{DarMed96}.
The notion of double extension was also defined for cosymplectic Lie algebras (see \cite{BouMan24Lef}).
In this subsection, we define the notion of double extension for NDB structures.
This notion can be regarded as an analogue of the notions of double extension for quadratic and cosymplectic Lie algebras.

Let $\mathfrak{g}$ be a Lie algebra with an NDB structure $\phi=\phi_S+\phi_A$.
Take $\theta\in\bigwedge^2 \mathfrak{g}^*$ with $d\theta=0$ and take $D\in\Der(\mathfrak{g})$ and $\lambda \in \mathfrak{g}^*$.
Suppose that $\theta, D$, and $\lambda$ satisfy
\begin{align}
&\lambda ([x,y])=\theta(D(x),y)+\theta(x,D(y)) \quad  \text{and}\label{4-4}\\
&\phi_S(D(x), y)+\phi_S(x,D(y))=0 \label{4-5}
\end{align}
for any $x,y\in\mathfrak{g}$.
Define a bracket $[\, \cdot \, , \, \cdot \, ]^{\tilde{}}$ on $\tilde{\mathfrak{g}}={\R}\langle e \rangle \oplus \mathfrak{g} \oplus {\R} \langle d\rangle $ by
\begin{align*}
[x,y]^{\tilde{}}&=[x,y] +\theta(x,y) e, \\
[d,x]^{\tilde{}}&=D(x)+\lambda (x) e, \quad \text{and} \\
[e,\tilde{\mathfrak{g}}]^{\tilde{}}&=0
\end{align*}
for $x,y\in\ \mathfrak{g}$.
Since $d\theta=0$, $D\in\Der(\mathfrak{g})$, and the equation \eqref{4-4} holds, $[\, \cdot \, , \, \cdot \, ]^{\tilde{}}$ is a Lie bracket on $\tilde{\mathfrak{g}}$.

We denote by $\mathfrak{g}_{\theta}$ the Lie algebra obtained as  the central extension of $\mathfrak{g}$ by ${\R}\langle e\rangle $ determined by $\theta$.
Then the Lie algebra $\tilde{\mathfrak{g}}$ is the semidirect sum ${\R}\langle d\rangle \ltimes_{\tilde{D}}\mathfrak{g}_{\theta}$, where  $\tilde{D}\in\Der(\mathfrak{g}_{\theta})$ is given by $\tilde{D}(x)=D(x)+\lambda(x) e$ for $x\in\mathfrak{g}$  and $\tilde{D}(e)=0$.

Take $k,l\in{\R}$ with $k\not= \pm l$ and $\alpha\in \mathfrak{g}^*$.
Suppose that $k,l$, and $\alpha$ satisfy the following equations:
\begin{align}
&\alpha([x,y])+\theta(x,y) k-\phi_S(y,D(x))=0. \label{4-6}\\ 
&\alpha(D(x))+k\lambda(x)=0. \label{4-7} \\
&\phi_A(D(x),y)+\phi_A(x,D(y))+\alpha([x,y])+l\theta(x,y)+\phi_S(x,D(y))=0. \label{4-8}
\end{align}

\begin{prop}
Take $k'\in{\R}$ and define a bilinear form $\tilde{\phi}=\tilde{\phi}_S+\tilde{\phi}_A$ on $\tilde{\mathfrak{g}}$ by
\begin{equation}
\tilde{\phi}_S=\begin{pmatrix} 0& 0& k\\ 0&\phi_S &\alpha \\ k&\alpha &k' \end{pmatrix} \quad \text{and}\quad  \tilde{\phi}_A=\begin{pmatrix}0&0&l\\0&\phi_A &\alpha \\ -l &-\alpha &0 \end{pmatrix} \label{df_double}
\end{equation}
with respect to the vector space decomposition $\tilde{\mathfrak{g}}={\R}\langle e\rangle \oplus \mathfrak{g} \oplus {\R} \langle d\rangle$.
Then $\tilde{\phi}$ is an NDB structure on $\tilde{\mathfrak{g}}$.
\end{prop}

\begin{proof}
Since $k\not= \pm l$, $\tilde{\phi}$ is non-degenerate.

For $x,y,z\in\mathfrak{g}$, $\tilde{\phi}_A$ and $\tilde{\phi}_S$ satisfy
\begin{align*}
&\tilde{\phi}_A([x,y]^{\tilde{}},z)+\tilde{\phi}_A([y,z]^{\tilde{}},x)+\tilde{\phi}_A([z,x]^{\tilde{}},y)\\
&=\tilde{\phi}_A([x,y]+\theta(x,y)e,z)+\tilde{\phi}_A([y,z]+\theta(y,z)e,x)+\tilde{\phi}_A([z,x]+\theta(z,x)e,y) \notag\\
&=\phi_A([x,y], z)+\phi_A([y,z], x)+\phi_A([z,x], y) \quad \text{and}\notag\\
&\tilde{\phi}_S(x,[y,z]^{\tilde{}})=\tilde{\phi}_S(x,[y,z]+\theta(y,z)e)=\phi_S(x,[y,z]).
\end{align*}
Since $\phi_S$ and $\phi_A$ satisfy \eqref{eq2}, we have
$$
\tilde{\phi}_A([x,y]^{\tilde{}},z)+\tilde{\phi}_A([y,z]^{\tilde{}},x)+\tilde{\phi}_A([z,x]^{\tilde{}},y)=\tilde{\phi}_S(x,[y,z]^{\tilde{}})
$$
for $x,y,z\in\mathfrak{g}$.

For $x,y\in\mathfrak{g}$, the following equations hold:
\begin{align*}
&\tilde{\phi}_A([x,y]^{\tilde{}}, e)+\tilde{\phi}_A([y,e]^{\tilde{}},x)+\tilde{\phi}_A([e,x]^{\tilde{}},y)=0.\\
&\tilde{\phi}_S(x,[y,e]^{\tilde{}})=\tilde{\phi}_S(y,[e,x]^{\tilde{}})=\tilde{\phi}_S(e,[x,y]^{\tilde{}})=0.\\
&\tilde{\phi}_A([x,e]^{\tilde{}},e)+\tilde{\phi}_A([e,e]^{\tilde{}},x)+\tilde{\phi}_A([e,x]^{\tilde{}},e)=0.\\
&\tilde{\phi}_S(x,[e,e]^{\tilde{}})=\tilde{\phi}_S(e,[e,x]^{\tilde{}})=\tilde{\phi}_S(e,[x,e]^{\tilde{}})=0.
\end{align*}
These equations yield the following equations:
\begin{align*}
&\tilde{\phi}_A([x,y]^{\tilde{}}, e)+\tilde{\phi}_A([y,e]^{\tilde{}},x)+\tilde{\phi}_A([e,x]^{\tilde{}},y)=\tilde{\phi}_S(x,[y,e]^{\tilde{}})=\tilde{\phi}_S(y,[e,x]^{\tilde{}})=\tilde{\phi}_S(e,[x,y]^{\tilde{}}).\\
&\tilde{\phi}_A([x,e]^{\tilde{}},e)+\tilde{\phi}_A([e,e]^{\tilde{}},x)+\tilde{\phi}_A([e,x]^{\tilde{}},e)=\tilde{\phi}_S(x,[e,e]^{\tilde{}})=\tilde{\phi}_S(e,[e,x]^{\tilde{}})=\tilde{\phi}_S(e,[x,e]^{\tilde{}}).
\end{align*}

By using the equations \eqref{4-5}, \eqref{4-6}, and \eqref{4-8} we obtain the following equations: 
\begin{align*}
&\tilde{\phi}_A([x,y]^{\tilde{}},d)+\tilde{\phi}_A([y,d]^{\tilde{}},x)+\tilde{\phi}_A([d,x]^{\tilde{}},y)\\
&=\tilde{\phi}_A([x,y]+\theta(x,y)e,d)+\tilde{\phi}_A(-D(y)-\lambda(y)e,x)+\tilde{\phi}_A(D(x)+\lambda(x)e, y)\\
&=\alpha([x,y])+\theta(x,y)l-\phi_A(D(y),x)+\phi_A(D(x),y) \\
&=-\phi_S(x,D(y)).\\
&\tilde{\phi}_S(x,[y,d]^{\tilde{}})
=\tilde{\phi}_S(x,-D(y)-\lambda(y)e)
=-\phi_S(x,D(y)).\\
&\tilde{\phi}_S(y,[d,x]^{\tilde{}})=\phi_S(y,D(x)).\\
&\tilde{\phi}_S(d,[x,y]^{\tilde{}})
=\tilde{\phi}_S(d,[x,y]+\theta(x,y)e)
=\alpha([x,y])+k\theta(x,y)
=\phi_S(y,D(x)).
\end{align*}
These equations imply 
\begin{align*}
&\tilde{\phi}_A([x,y]^{\tilde{}},d)+\tilde{\phi}_A([y,d]^{\tilde{}},x)+\tilde{\phi}_A([d,x]^{\tilde{}},y)\\
&=\tilde{\phi}_S(x,[y,d]^{\tilde{}})=\tilde{\phi}_S(y,[d,x]^{\tilde{}})=\tilde{\phi}_S(d,[x,y]^{\tilde{}}).
\end{align*}

Equations
\begin{align*}
&\tilde{\phi}_A([x,d]^{\tilde{}},d)+\tilde{\phi}_A([d,d]^{\tilde{}},x)+\tilde{\phi}_A([d,x]^{\tilde{}},d)=0, \\
&\tilde{\phi}_S(x,[d,d]^{\tilde{}})=0,\\
&\tilde{\phi}_S(d,[d,x]^{\tilde{}})=\tilde{\phi}_S(d, D(x)+\lambda(x)e)=\alpha(D(x))+k\lambda(x), \quad \text{and}\\
&\tilde{\phi}_S(d,[x,d]^{\tilde{}})=-\alpha(D(x))-k\lambda(x),
\end{align*}
and the equation \eqref{4-7} imply
\begin{align*}
&\tilde{\phi}_A([x,d]^{\tilde{}},d)+\tilde{\phi}_A([d,d]^{\tilde{}},x)+\tilde{\phi}_A([d,x]^{\tilde{}},d)\\
&=\tilde{\phi}_S(x,[d,d]^{\tilde{}})=\tilde{\phi}_S(d,[d,x]^{\tilde{}})=\tilde{\phi}_S(d,[x,d]^{\tilde{}}).
\end{align*}

Equations
\begin{align*}
&\tilde{\phi}_A([x,e]^{\tilde{}},d)+\tilde{\phi}_A([e,d]^{\tilde{}},x)+\tilde{\phi}_A([d,x]^{\tilde{}},e)=0 \quad\text{and}\\
&\tilde{\phi}_S(x,[e,d]^{\tilde{}})=\tilde{\phi}_S(e,[d,x]^{\tilde{}})=\tilde{\phi}_S(d,[x,e]^{\tilde{}})=0
\end{align*}
imply
\begin{align*}
&\tilde{\phi}_A([x,e]^{\tilde{}},d)+\tilde{\phi}_A([e,d]^{\tilde{}},x)+\tilde{\phi}_A([d,x]^{\tilde{}},e)\\
&=\tilde{\phi}_S(x,[e,d]^{\tilde{}})=\tilde{\phi}_S(e,[d,x]^{\tilde{}})=\tilde{\phi}_S(d,[x,e]^{\tilde{}}).
\end{align*}

Therefore, $\tilde{\phi}_S$ and $\tilde{\phi}_A$ satisfy the equation \eqref{eq2}.
\end{proof}

\begin{df}\label{def_double_extension}
We call the NDB Lie algebra $(\tilde{\mathfrak{g}}, \tilde{\phi})$  defined by \eqref{df_double} a double extension of $(\mathfrak{g}, \phi)$ by $(\theta, D, \lambda, \alpha, k ,k',l)$.
\end{df}

Conversely, an NDB Lie algebra $(\tilde{\mathfrak{g}}, \tilde{\phi})$ satisfying a certain condition is obtained by double extension.

\begin{thm}\label{thm_double_extension}
Let $\tilde{\phi}=\tilde{\phi}_S+\tilde{\phi}_A$ be an NDB structure on an $n$-dimensional Lie algebra $\tilde{\mathfrak{g}}$.
Suppose that there exists a nonzero element $e\in c(\tilde{\mathfrak{g}})$ such that $\tilde{\phi}(e,e)=0$ and 
$$
I^{\perp}=\{ x\in\tilde{\mathfrak{g}}\mid \tilde{\phi}(e,x)=\tilde{\phi}(x,e)=0\}
$$
is a codimension-one subspace of $\tilde{\mathfrak{g}}$, where $c(\tilde{\mathfrak{g}})$ is the center of $\tilde{\mathfrak{g}}$ and $I={\R}\langle e\rangle$.
Then there exists an NDB Lie algebra $(\mathfrak{g}, \phi)$ such that $(\tilde{\mathfrak{g}}, \tilde{\phi})$ is obtained by double extension of $(\mathfrak{g}, \phi)$ by $(\theta, D, \lambda, \alpha, k ,k',l)$.
\end{thm}

\begin{proof}
Since $e\in c(\tilde{\mathfrak{g}})$, by Lemma \ref{lem1}, $I^{\perp}$ is an ideal of $\tilde{\mathfrak{g}}$.

For any $x\in I^{\perp}$, $\tilde{\phi}_S$ and $\tilde{\phi}_A$ satisfy $\tilde{\phi}_S(x,e)+\tilde{\phi}_A(x,e)=0$ and $\tilde{\phi}_S(e,x)+\tilde{\phi}_A(e,x)=0$.
Since $\tilde{\phi}_S(x,e)=\tilde{\phi}_S(e,x)$ and $\tilde{\phi}_A(x,e)=-\tilde{\phi}_A(e,x)$, we have $\tilde{\phi}_S(x,e)=\tilde{\phi}_S(e,x)=0$ and $\tilde{\phi}_A(x,e)=\tilde{\phi}_A(e,x)=0$.

Take $d\in \tilde{\mathfrak{g}}\setminus I^{\perp}$ and put $\mathfrak{h}={\R}\langle e\rangle \oplus {\R}\langle d \rangle$.
Consider the subspace 
\begin{align*}
\mathfrak{h}^{\perp}_R
&=\{x\in\tilde{\mathfrak{g}}\mid \tilde{\phi}(z,x)=0 \text{ for any }z\in\mathfrak{h} \}\\
&=\{x\in\tilde{\mathfrak{g}}\mid \tilde{\phi}(e,x)=\tilde{\phi}(d,x)=0 \}.
\end{align*}
of $\tilde{\mathfrak{g}}$.
\begin{lem}\label{lem4-1}
The subspace $\mathfrak{h}^{\perp}_R$ is of codimension 2 in $\tilde{\mathfrak{g}}$ and $\tilde{\mathfrak{g}}={\R}\langle e \rangle \oplus \mathfrak{h}^{\perp}_R\oplus {\R} \langle d \rangle$ as a vector space.
\end{lem}

\begin{proof}
Since $\codim I^{\perp}=1$, $\tilde{\mathfrak{g}}=I^{\perp}\oplus {\R}\langle d\rangle$.
Since $\tilde{\phi}$ is non-degenerate and $\tilde{\phi}(e,I^{\perp})=\tilde{\phi}(I^{\perp},e)=0$, we have $\tilde{\phi}(e,d)\not= 0$ and $\tilde{\phi}(d,e)\not= 0$.
Hence $\tilde{\phi}|_{\mathfrak{h}}$ is non-degenerate.

Let $\iota_{\mathfrak{h}}\colon\mathfrak{h}\to\tilde{\mathfrak{g}}$ be the inclusion map and consider the map $\iota_{\mathfrak{h}}^*\circ \tilde{\phi}^{\flat}_{R}\colon \tilde{\mathfrak{g}}\to \mathfrak{h}^*$.
Then $\iota_{\mathfrak{h}}^*\circ \tilde{\phi}^{\flat}_{R}$ is surjective and $\mathfrak{h}^{\perp}_R=\Ker(\iota_{\mathfrak{h}}^*\circ \tilde{\phi}^{\flat}_R)$.
Hence $\codim \mathfrak{h}^{\perp}_R=2$.

Since $\tilde{\phi}|_{\mathfrak{h}}$ is non-degenerate, $\mathfrak{h}\cap \mathfrak{h}^{\perp}_R=\{0\}$.
Therefore we have
$$
\tilde{\mathfrak{g}}=\mathfrak{h}\oplus \mathfrak{h}^{\perp}_R={\R}\langle e\rangle \oplus \mathfrak{h}^{\perp}_R\oplus {\R}\langle d\rangle.
$$ 
\end{proof}

\begin{lem}\label{lem4-2}
$I^{\perp}={\R}\langle e \rangle \oplus \mathfrak{h}^{\perp}_R$.
\end{lem}

\begin{proof}
Consider the subspace
$$
I^{\perp}_R=\{x\in\tilde{\mathfrak{g}}\mid \tilde{\phi}(e,x)=0\}
$$
of $\tilde{\mathfrak{g}}$.
Since $\codim I^{\perp}=1$ and $I^{\perp}\subset I^{\perp}_R$, we have $I^{\perp}_R=I^{\perp}$ or $I^{\perp}_R=\mathfrak{g}$.
Since $\tilde{\phi}(e,d)\not= 0$, we have $d\not\in I^{\perp}_R$ and hence $I^{\perp}_R=I^{\perp}$.
Thus $I\subset \mathfrak{h}$ implies that $\mathfrak{h}^{\perp}_R\subset I^{\perp}_R=I^{\perp}$.

Since $\codim I^{\perp}=1, \codim \mathfrak{h}^{\perp}_R=2, e\in I^{\perp}$, and $e\not\in\mathfrak{h}^{\perp}_R$, we have $I^{\perp}=\mathfrak{h}^{\perp}_R\oplus {\R}\langle e\rangle$.
\end{proof}

By Lemmas \ref{lem4-1} and \ref{lem4-2}, we have $I^{\perp}={\R}\langle e\rangle \oplus \mathfrak{h}^{\perp}_R$ and $\tilde{\mathfrak{g}}={\R}\langle e\rangle \oplus \mathfrak{h}^{\perp}_R\oplus {\R}\langle d\rangle$ as vector spaces.
Hence $\tilde{\phi}$ is of the form 
$$
\tilde{\phi}=\begin{pmatrix} 0&0&*\\ 0 &* &*\\ *&*& *\end{pmatrix}
$$
with respect to the decomposition $\tilde{\mathfrak{g}}={\R}\langle e\rangle \oplus \mathfrak{h}^{\perp}_R\oplus {\R}\langle d\rangle$.
Therefore $\tilde{\phi}|_{\mathfrak{h}^{\perp}_R\times \mathfrak{h}^{\perp}_R}$ is non-degenerate.

Let $\mathfrak{g}$ denote the quotient Lie algebra $I^{\perp}/I$, which is isomorphic to $\mathfrak{h}^{\perp}_R$ as a vector space.
Define a bilinear form $\phi$ on $\mathfrak{g}$ by $\phi(p(x), p(y))=\tilde{\phi}(x,y)$, where $p\colon I^{\perp}\to \mathfrak{g}$ is the projection.
Then $\phi$ is well-defined.
Moreover, $\phi$ is non-degenerate since $\tilde{\phi}|_{\mathfrak{h}^{\perp}_R \times \mathfrak{h}^{\perp}_R}$ is non-degenerate.
The equation \eqref{eq2} for $\tilde{\phi}$ yields the equation \eqref{eq2} for $\phi$ by Lemma \ref{lem1}.
Hence $\phi$ is an NDB structure on $\mathfrak{g}$.

Since $I={\R}\langle e\rangle$ is a central ideal of $I^{\perp}$, there exists a cocycle $\theta\in \bigwedge^2 \mathfrak{g}^*$ such that $I^{\perp}$ is isomorphic to the central extension $\mathfrak{g}_{\theta}$ of $\mathfrak{g}$ by ${\R}\langle e\rangle $.
Under the identification of $I^{\perp}$ with $\mathfrak{g}_{\theta}$, we put $\tilde{D}=\ad(d)\in\Der(\mathfrak{g}_{\theta})$.
Since $e\in c(\tilde{\mathfrak{g}})$ and $\mathfrak{g}_{\theta}={\R}\langle e\rangle \oplus \mathfrak{g}$ as a vector space, $\tilde{D}$ satisfies $\tilde{D}(e)=0$ and $\tilde{D}(x)$ can be written as $\tilde{D}(x)=D(x)+\lambda(x) e$ for $x\in\mathfrak{g}$, where $D\in\mathfrak{gl}(\mathfrak{g})$ and $\lambda\in\mathfrak{g}^*$.
Then $\tilde{\mathfrak{g}}$ is isomorphic to ${\R}\langle d\rangle \ltimes_{\tilde{D}}\mathfrak{g}_{\theta}$.
Since $\tilde{D}\in\Der(\mathfrak{g}_{\theta})$, $D$ and $\lambda$ satisfy the equation \eqref{4-4} and $D\in\Der(\mathfrak{g})$.

Put $k=\tilde{\phi}_S(e,d), \alpha=(\tilde{\phi}_S)^{\flat}_L(d)|_{\mathfrak{g}}, k'=\tilde{\phi}_S(d,d), l=\tilde{\phi}_A(e,d)$, and $\beta=(\tilde{\phi}_A)^{\flat}_L(d)|_{\mathfrak{g}}$.
Since $\tilde{\phi}(d,x)=0$ for $x\in\mathfrak{g}$, we have $\beta=-\alpha$.
Then $\tilde{\phi}=\tilde{\phi}_S+\tilde{\phi}_A$ is of the form
$$
\tilde{\phi}_S=\begin{pmatrix}0&0&k\\0&\phi_S&\alpha\\ k&\alpha&k'\end{pmatrix} \quad \text{and}\quad  \tilde{\phi}_A=\begin{pmatrix}0&0&l\\0&\phi_A&\alpha\\-l &-\alpha &0\end{pmatrix}
$$
with respect to the decomposition $\tilde{\mathfrak{g}}={\R}\langle e\rangle \oplus \mathfrak{g}\oplus {\R}\langle d\rangle$.

Since $\tilde{\phi}_S$ is $\ad$-invariant, $\tilde{\phi}_S$ satisfies
\begin{align*}
&\tilde{\phi}_S([d,x\tilde{]},y)+\tilde{\phi}_S(x,[d,y\tilde{]})=0,\\
&\tilde{\phi}_S([x,y\tilde{]}, d)+\tilde{\phi}_S(y,[x,d\tilde{]})=0, \quad \text{and}\\
&\tilde{\phi}_S([d,x\tilde{]},d)+\tilde{\phi}_S(x, [d,d\tilde{]})=0.
\end{align*}
These equations imply equations \eqref{4-5}, \eqref{4-6}, and \eqref{4-7}.

By the equation \eqref{eq2}, we have
$$
\tilde{\phi}_A([x,y\tilde{]}, d)+\tilde{\phi}_A([y,d\tilde{]}, x)+\tilde{\phi}_A([d,x\tilde{]},y) = \tilde{\phi}_S(x,[y,d\tilde{]}).
$$
This equation implies the equation \eqref{4-8}.

Therefore, $\tilde{\phi}$ is obtained by double extension of $(\mathfrak{g}, \phi)$ by $(\theta, D, \lambda, \alpha, k ,k',l)$.
\end{proof}

Let $\psi$ be a non-degenerate $\ad$-invariant symmetric form on a Lie algebra $\mathfrak{g}$.
Then the double extension $(\tilde{\mathfrak{g}}, \tilde{\psi})$ of $(\mathfrak{g}, \psi)$ by ${\R}$ using a skew-symmetric derivation $D\in\Der(\mathfrak{g}, \psi)$ is defined as follows (see \cite{MedRev85}, \cite{Ova16}):
$\tilde{\mathfrak{g}}={\R}\langle e\rangle \oplus \mathfrak{g} \oplus {\R}\langle d\rangle$ with Lie bracket 
$$
[f_1 e+x_1+s_1 d, f_2 e+x_2+s_2d]^{\tilde{}}=\psi(D(x_1), x_2) e+ [x_1,x_2]+s_1D(x_2)-s_2D(x_1),
$$
where $f_i, s_i\in {\R}$ and $x_i\in\mathfrak{g}$ for $i=1,2$.
The non-degenerate $\ad$-invariant symmetric form $\tilde{\psi}$ on $\tilde{\mathfrak{g}}$ is defined by
$$
\tilde{\psi}(f_1 e+x_1+s_1 d, f_2 e+x_2+s_2d)=\psi(x_1,x_2)+k's_1s_2+ f_1s_2+f_2s_1,
$$
where $k'\in {\R}$.
Hence, $\tilde{\mathfrak{g}}={\R}\langle d\rangle\ltimes_{\tilde{D}}{\mathfrak{g}}_{\theta}$ and
$$
\tilde{\psi}=\begin{pmatrix} 0&0&1\\ 0&\psi&0\\ 1&0&k' \end{pmatrix}
$$
with respect to the vector space decomposition $\tilde{\mathfrak{g}}={\R}\langle e \rangle \oplus \mathfrak{g}\oplus{\R}\langle d \rangle$, where $\theta\in\bigwedge^2\mathfrak{g}^*$ is defined by $\theta(x,y)=\psi(D(x), y)$ and $\tilde{D}\in\Der(\mathfrak{g}_{\theta})$ is defined by $\tilde{D}(x)=D(x)$ for $x\in\mathfrak{g}$ and $\tilde{D}(e)=0$.

In the double extension construction for NDB structures, if we suppose that $\phi=\phi_S$ is $\ad$-invariant but does not necessarily satisfy the equation \eqref{eq2} and if equations \eqref{4-4}, \eqref{4-5}, \eqref{4-6}, and \eqref{4-7} hold, then the symmetric part $\tilde{\phi}_S$ of $\tilde{\phi}$ defined by \eqref{df_double} is an $\ad$-invariant symmetric form on $\tilde{\mathfrak{g}}$.
The double extension of the quadratic Lie algebra $(\mathfrak{g}, \psi)$ by ${\R}$ using $D$ coincides with the pair $(\tilde{\mathfrak{g}}, \tilde{\phi}_S)$ defined above with respect to $(\theta, D, 0,0,1,k',0)$.
 
El Bourkadi and Mansouri \cite{BouMan24Lef} considered the double extension $\tilde{\mathfrak{g}}$ of a Lie algebra $\mathfrak{g}$ by $(\tilde{D}, \theta)$, where $\theta\in\bigwedge^2\mathfrak{g}^*$ is a $2$-cocycle and $\tilde{D}\in\Der(\mathfrak{g}_\theta)$.
For a Lie algebra $\mathfrak{g}$ with a cosymplectic structure $(\alpha,\omega)$, they considered a $1$-form $\tilde{\alpha}$ and a $2$-form $\tilde{\omega}$ on $\tilde{\mathfrak{g}}$ such that $\tilde{\alpha}=\alpha$ on $\mathfrak{g}$, $\tilde{\alpha}(e)=0$,  and $\tilde{\omega}=\omega+d^*\wedge e^*$.
They gave a necessary and sufficient condition for $(\tilde{\alpha}, \tilde{\omega})$ to be a cosymplectic structure on $\tilde{\mathfrak{g}}$ (see Theorem 3.8 in \cite{BouMan24Lef}).
For $(\tilde{\alpha}, \tilde{\omega})$, the associated NDB structure $\tilde{\phi}=\tilde{\phi}_S+\tilde{\phi}_A$ is of the form
$$
\tilde{\phi}_S=\begin{pmatrix} 0&0&0\\ 0&\alpha\otimes\alpha &\tilde{\alpha}(d)\alpha\\0&\tilde{\alpha}(d)\alpha&\tilde{\alpha}(d)^2  \end{pmatrix} \quad \text{and} \quad \tilde{\phi}_A=\begin{pmatrix} 0&0&-1 \\ 0&\omega& 0\\1&0&0\end{pmatrix}
$$
with respect to the vector space decomposition $\tilde{\mathfrak{g}}={\R}\langle e\rangle \oplus\mathfrak{g}\oplus {\R}\langle d \rangle$.
If $\tilde{D}$ is of the form $\tilde{D}=\begin{pmatrix} D&0\\ \lambda &0\end{pmatrix}$ with respect to the decomposition $\mathfrak{g}_{\theta}=\mathfrak{g}\oplus{\R}\langle  e\rangle$ and $\tilde{\alpha}(d)=0$, then the NDB Lie algebra associated with $(\tilde{\alpha}, \tilde{\omega})$ coincides with the double extension of $(\mathfrak{g}, \phi)$ by $(\theta, D, \lambda, 0,0,0,-1)$, where $\phi=\alpha\otimes \alpha+\omega$ is the NDB structure associated with $(\alpha,\omega)$.

\section*{acknowledgement}
This work was supported by JSPS KAKENHI under Grant Number JP26K06819.

\bibliography{reference}

@article{Bor97,
	author = {Bordemann, M.},
	fjournal = {Acta Mathematica Universitatis Comenianae. New Series},
	issn = {0862-9544},
	journal = {Acta Math. Univ. Comen., New Ser.},
	language = {English},
	number = {2},
	pages = {151--201},
	title = {Nondegenerate invariant bilinear forms on nonassociative algebras},
	url = {https://eudml.org/doc/122700},
	volume = {66},
	year = {1997},
	zbl = {1014.17003},
	zbmath = {1332348}}

@misc{Kat25-2,
	archiveprefix = {arXiv},
	author = {Naoki Kato},
	eprint = {2510.14610},
	note = {arXiv:2510.14610},
	primaryclass = {math.RA},
	title = {Infinitely many left-symmetric structures on nilpotent {Lie} algebras},
	url = {https://arxiv.org/abs/2510.14610},
	year = {2025}}

@article{Gic05,
	archiveprefix = {arXiv},
	author = {V. M. Gichev},
	eprint = {math/0512558},
	note = {arXiv:math/0512558},
	primaryclass = {math.DG},
	title = {On complete affine structures in {Lie} groups},
	url = {https://arxiv.org/abs/math/0512558},
	year = {2005}}

@article{Cor16,
	author = {Cornulier, Yves},
	doi = {10.1515/forum-2014-0048},
	issn = {1435-5337},
	journal = {Forum Mathematicum},
	month = Sept,
	number = {1},
	pages = {101--128},
	publisher = {Walter de Gruyter GmbH},
	title = {On the {Koszul} map of {Lie} algebras},
	url = {http://dx.doi.org/10.1515/forum-2014-0048},
	volume = {28},
	year = {2016}}

@article{DarMed96,
	author = {Dardi{\'e}, Jean-Michel and Medina, Alberto},
	doi = {10.1006/aima.1996.0009},
	issn = {0001-8708},
	journal = {Advances in Mathematics},
	number = {2},
	pages = {208--227},
	publisher = {Elsevier BV},
	title = {Double Extension Symplectique d'un Groupe de {Lie} Symplectique},
	url = {http://dx.doi.org/10.1006/aima.1996.0009},
	volume = {117},
	year = {1996}}

@article{BouMan24Lef,
	author = {El Bourkadi, Said and Mansouri, Mohammed W.},
	doi = {10.5802/jolt.1335},
	issn = {0949-5932},
	journal = {Journal of Lie Theory},
	number = {2},
	pages = {249--265},
	publisher = {MathDoc/Centre Mersenne},
	title = {Left-symmetric products on cosymplectic {Lie} algebras},
	url = {http://dx.doi.org/10.5802/jolt.1335},
	volume = {34},
	year = {2024}}

@article{KhaGozMed04,
	author = {Khakimdjanov, Yu. and Goze, M. and Medina, A.},
	doi = {10.1016/j.difgeo.2003.12.006},
	issn = {0926-2245},
	journal = {Differential Geometry and its Applications},
	number = {1},
	pages = {41--54},
	publisher = {Elsevier BV},
	title = {Symplectic or contact structures on {Lie} groups},
	url = {http://dx.doi.org/10.1016/j.difgeo.2003.12.006},
	volume = {21},
	year = {2004}}

@article{Ova16,
	author = {Ovando, Gabriela P.},
	journal = {Rend. Semin. Mat., Univ. Politec. Torino},
	number = {1},
	pages = {243--268},
	title = {Lie algebras with ad-invariant metrics. {A} survey - guide},
	volume = {74},
	year = {2016}}

@article{Bur06,
	author = {Burde, Dietrich},
	doi = {10.2478/s11533-006-0014-9},
	fjournal = {Central European Journal of Mathematics},
	issn = {1895-1074},
	journal = {Cent. Eur. J. Math.},
	language = {English},
	number = {3},
	pages = {323--357},
	title = {Left-symmetric algebras, or pre-{Lie} algebras in geometry and physics},
	volume = {4},
	year = {2006},
	zbl = {1151.17301},
	zbmath = {5209438}}

@article{Chu74,
	author = {Chu, Bon-Yao},
	doi = {10.2307/1996932},
	fjournal = {Transactions of the American Mathematical Society},
	issn = {0002-9947},
	journal = {Trans. Am. Math. Soc.},
	language = {English},
	pages = {145--159},
	title = {Symplectic homogeneous spaces},
	volume = {197},
	year = {1974},
	zbl = {0261.53039},
	zbmath = {3411826}}

@article{Hel79,
	author = {Helmstetter, Jacques},
	doi = {10.5802/aif.764},
	fjournal = {Annales de l'Institut Fourier},
	issn = {0373-0956},
	journal = {Ann. Inst. Fourier},
	language = {French},
	number = {4},
	pages = {17--35},
	title = {Radical d'une alg{\`e}bre sym{\'e}trique {\`a} gauche},
	volume = {29},
	year = {1979},
	zbl = {0403.16020},
	zbmath = {3625581}}

@article{Kim86,
	author = {Kim, Hyuk},
	doi = {10.4310/jdg/1214440553},
	fjournal = {Journal of Differential Geometry},
	issn = {0022-040X},
	journal = {J. Differ. Geom.},
	language = {English},
	pages = {373--394},
	title = {Complete left-invariant affine structures on nilpotent {Lie} groups},
	volume = {24},
	year = {1986},
	zbl = {0591.53045},
	zbmath = {3949235}}

@article{MedRev85,
	author = {Medina, Alberto and Revoy, Philippe},
	doi = {10.24033/asens.1496},
	fjournal = {Annales Scientifiques de l'{\'E}cole Normale Sup{\'e}rieure. Quatri{\`e}me S{\'e}rie},
	issn = {0012-9593},
	journal = {Ann. Sci. {\'E}c. Norm. Sup{\'e}r. (4)},
	language = {French},
	pages = {553--561},
	title = {Alg{\`e}bres de {Lie} et produit scalaire invariant},
	volume = {18},
	year = {1985},
	zbl = {0592.17006},
	zbmath = {3950771}}

@inproceedings{MedRev91,
	author = {Medina, Alberto and Revoy, Philippe},
	booktitle = {{Symplectic Geometry, Groupoids, and Integrable Systems, S{\'e}minaire Sud Rhodanien de G{\'e}om{\'e}trie {\`a} Berkeley}},
	editor = {Pierre Dazord and Alan Weinstein},
	journal = {Symplectic geometry, groupoids, and integrable systems, {S{\'e}min}. {Sud}- {Rhodan}. {Geom}. {VI}, {Berkeley}/{CA} ({USA}) 1989, {Math}. {Sci}. {Res}. {Inst}. {Publ}. 20},
	pages = {247--266},
	publisher = {Springer New York},
	series = {Mathematical Sciences Research Institute Publications},
	title = {Groupes de {Lie} {\`a} structure symplectique invariante},
	year = {1991},
	zbl = {0754.53027},
	zbmath = {16707}}

@article{Ova06,
	author = {Ovando, Gabriela},
	fjournal = {Beitr{\"a}ge zur Algebra und Geometrie},
	issn = {0138-4821},
	journal = {Beitr. Algebra Geom.},
	language = {English},
	number = {2},
	pages = {419--434},
	title = {Four dimensional symplectic {Lie} algebras},
	volume = {47},
	year = {2006},
	zbl = {1155.53042},
	zbmath = {5117812}}

@article{MedSalGir16,
	author = {Medina, A. and Saldarriaga, O. and Giraldo, H.},
	doi = {10.1016/j.jalgebra.2016.02.007},
	fjournal = {Journal of Algebra},
	issn = {0021-8693},
	journal = {J. Algebra},
	language = {English},
	pages = {183--208},
	title = {Flat affine or projective geometries on {Lie} groups},
	volume = {455},
	year = {2016},
	zbl = {1339.57043},
	zbmath = {6566604}}
\par\noindent{\scshape \small
Faculty of Liberal Arts and Sciences,
Chukyo University, \\
101 Tokodachi, Kaizu-cho, Toyota-shi, Aichi 470-0393, Japan.}
\par\noindent{knaoki@lets.chukyo-u.ac.jp}

\end{document}